\documentclass[12pt,a4paper]{article}
\usepackage{amsmath}
\usepackage{amssymb}
\usepackage{amssymb}
\usepackage{cases}
\usepackage{bbm}
\usepackage{mathrsfs}
\usepackage{graphicx}
\usepackage{amsfonts}
\usepackage{enumerate}
\usepackage{theorem}
\usepackage[pdfstartview=FitH]{hyperref}
\makeatletter
\def\tank#1{\protected@xdef\@thanks{\@thanks
 \protect\footnotetext[0]{#1}}}
\def\bigfoot{

 \@footnotetext}
\makeatother

\newcommand{\ea}{\end{array}}

\allowdisplaybreaks

\newtheorem{theorem}{Theorem}[section]
\newtheorem{lem}{Lemma}[section]
\newtheorem{prp}[theorem]{Proposition}
\newtheorem{thm}[theorem]{Theorem}
\newtheorem{cor}[theorem]{Corollary}
\newtheorem{dfn}[theorem]{Definition}

\def\beq{\begin{equation}}
\def\nneq{\end{equation}}

\def\bthm{\begin{thm}}
\def\nthm{\end{thm}}

\def\blem{\begin{lem}}
\def\nlem{\end{lem}}
\def\bprf{\begin{proof}}
\def\nprf{\end{proof}}
\def\bprop{\begin{prop}}
\def\nprop{\end{prop}}
\def\brmk{\begin{rem}}
\def\nrmk{\end{rem}}

\def\bexa{\begin{exa}}
\def\nexa{\end{exa}}
\def\bcor{\begin{cor}}
\def\ncor{\end{cor}}

\def\AA{\mathcal A}

\def\FF{\mathcal F}
\def\EE{\mathcal E}

\def\HH{\mathcal H}

\def\e{\varepsilon}

{\theorembodyfont{\rmfamily}
}

\title{Large deviations for stochastic models of two-dimensional second grade fluids
}

\footnotesize{\author{Jianliang Zhai $^{1,}$\thanks{zhaijl@ustc.edu.cn},\ \
 Tusheng Zhang$^{2,}$\thanks{Tusheng.Zhang@manchester.ac.uk}\\
 {\em $^1$ School of Mathematical Sciences,}\\
 {\em University of Science and Technology of China,}\\
 {\em Hefei, 230026, China}\\
 {\em $^2$ School of Mathematics, University of Manchester,}\\
 {\em Oxford Road, Manchester, M13 9PL, UK}\\
}
\date{}
\newenvironment{proof}{\par\noindent{\bf Proof:}}{\hspace*{\fill}$\blacksquare$\par}
\begin{document}
\maketitle
\noindent \textbf{Abstract:}
In this paper, we established a large deviation principle for stochastic models of incompressible second grade fluids.
The weak convergence method introduced by \cite{Budhiraja-Dupuis} plays an important role.


\vspace{4mm}


\vspace{3mm}
\section{Introduction}

In this paper, we are concerned with large deviation principles for stochastic models for the incompressible second grade
fluid which is a particular class of Non-Newtonian fluid. Let $\mathcal{O}$ be a connected, bounded open
subset of $\mathbb{R}^2$ with boundary $\partial \mathcal{O}$ of class $\mathcal{C}^3$. We consider
\begin{eqnarray}\label{01}
&  &d(u^\epsilon-\alpha \triangle u^\epsilon)+\Big(-\nu \triangle u^\epsilon+curl(u^\epsilon-\alpha \triangle u^\epsilon)\times u^\epsilon+\nabla\mathfrak{P}^\e\Big)dt\\\nonumber
& &=
F(u^\epsilon,t)dt+\sqrt{\epsilon}G(u^\epsilon,t)dW,\ \ \ in\ \mathcal{O}\times(0,T],
\end{eqnarray}
under the following condition
\begin{eqnarray}\label{02}
\left\{
 \begin{array}{llll}
 & \hbox{${\rm{div}}\ u^\epsilon=0\ \text{in}\ \mathcal{O}\times(0,T]$;} \\
 & \hbox{$u^\epsilon=0\ \text{in}\ \partial \mathcal{O}\times[0,T]$;} \\
 & \hbox{$u^\epsilon(0)=u_0\ \text{in}\ \mathcal{O}$,}
 \end{array}
\right.
\end{eqnarray}
where $u^\e=(u_1^\e,u_2^\e)$ and $\mathfrak{P}^\e$ represent the random velocity and modified pressure, respectively.
$W$ is an $m$-dimensional standard Brownian motion defined on a complete probability space $(\Omega,\mathcal{F},\{\mathcal{F}_t\}_{t\in[0,T]},P)$.

The interest in the investigation of the second grade fluids arises from the fact that it is an admissible
model of slow flow fluids, which contains a large class Non-Newtonian fluids such as industrial fluids, slurries, polymer melts, etc..
Furthermore, ``the second grade fluid has general and pleasant properties such as boundedness, stability, and exponential decay"(see \cite{Dunn-Fosdick}).
It also has interesting connections with many other fluid models, see \cite{Busuioc, Busuioc-Ratiu, Holm-Marsden-Ratiu, Holm-Marsden-Ratiu01, Iftimie, Shkoller2001, Shkoller1998} and references therein. For example, it can be taken as a generalization of the Navier-Stokes Equation.
Indeed they reduce to Navier-Stokes Equation when $\alpha=0$. Furthermore, it was shown in \cite{Iftimie} that the second grade fluids models are good approximations of the Navier-Stokes Equation.
We refer to \cite{Dunn-Fosdick, Dunn-Rajagopal, Fosdick-Rajagopal, Noll-truesdell} for a comprehensive theory of the second grade fluids.

Recently, the stochastic models of two-dimensional second grade fluids (\ref{01}) have been studied in \cite{RS-12}, \cite{RS-10-01}  and \cite{RS-10}, where the authors obtained the existence and uniqueness of solutions
and investigated the behavior of the solution as $\alpha\rightarrow0$. The martingale solution of
the system (\ref{01}) driven by L\'evy noise is studied in \cite{HRS}.

In the present work we are concerned with large deviation principles of the solutions
of the system (\ref{01}). Large deviations have applications in many areas,
such as in thermodynamics, statistical mechanics, information
theory and risk management, etc., see  \cite{Dembo-Zeitouni} \cite{Touchette}and reference therein. Large deviations for stochastic evolution equations and stochastic
 partial differential equations driven by Gaussian processes have been
investigated in many papers, see e.g.\ \cite{BDM08}, \cite{BDM10}, \cite{BDF12}, \cite{CW}, \cite{CR}, \cite{CM2}, \cite{Liu}, \cite{S}, \cite{Z}.
In this paper, we will apply the weak convergence approach introduced in \cite{Budhiraja-Dupuis}. This approach is mainly based on a variational representation formula for certain functionals of infinite dimensional Brownian Motion.
Technical difficulties arise when implementing weak convergence approach to the system (\ref{01}). One of them is to deal with the nonlinear term $\rm curl(u^\e-\alpha \triangle u^\e)$.

 The organization of this paper is as follows. In Section 2, we introduce some functional spaces and state some lemmas needed later.
  Section 3 is to formulate the hypotheses and to recall the theorem of existence of solutions for system (\ref{01}) obtained in \cite{RS-12}.
  The entire Section 4 is devoted to establishing the large deviation principle for system (\ref{01}).

\section{Preliminaries}

In this section, we will introduce functional spaces and preliminary facts needed later.

Let $1\leq p<\infty$, and $k$ a nonnegative integer. We denote by $L^p(\mathcal{O})$
and $W^{k,p}(\mathcal{O})$ the usual $L^p$ and Sobolev spaces, and write $W^{k,2}(\mathcal{O})=H^k(\mathcal{O})$. Let $W^{k,p}_0(\mathcal{O})$ be
the closure in $W^{k,p}(\mathcal{O})$ of $\mathcal{C}^\infty_c(\mathcal{O})$ the space of infinitely differentiable functions with compact supports in
$\mathcal{O}$. We denote $W^{k,2}_0(\mathcal{O})$ by $H_0^k(\mathcal{O})$. We endow the Hilbert space $H^1_0(\mathcal{O})$
with the scalar product
\begin{eqnarray}\label{H-01}
((u,v))=\int_\mathcal{O}\nabla u\cdot\nabla vdx=\sum_{i=1}^2\int_\mathcal{O}\frac{\partial u}{\partial x_i}\frac{\partial v}{\partial x_i}dx,
\end{eqnarray}
where $\nabla$ is the gradient operator. The norm $\|\cdot\|$ generated by this scalar product is equivalent to the usual norm of $W^{1,2}(\mathcal{O})$
in $H^1_0(\mathcal{O})$.

In what follows, we denote by $\mathbb{X}$ the space of $\mathbb{R}^2$-valued functions such that each component belongs to $X$. We introduce the spaces
\begin{eqnarray}\label{SP-01}
\mathcal{C}=\Big\{u\in[\mathcal{C}^\infty_c(\mathcal{O})]^2\ {\rm such \ that\  div}\ u=0\Big\},\nonumber\\
\mathbb{V}={\rm\ closure\ of}\ \mathcal{C} {\rm\ in}\ \mathbb{H}^1(\mathcal{O}),\\
\mathbb{H}={\rm\ closure\ of}\ \mathcal{C}\ {\rm in}\ \mathbb{L}^2(\mathcal{O}).\nonumber
\end{eqnarray}
We denote by $(\cdot,\cdot)$ and $|\cdot|$ the inner product and the norm induced by the inner product and the norm
in $\mathbb{L}^2(\mathcal{O})$ on $\mathbb{H}$, respectively. The inner product and the norm of $\mathbb{H}^1_0(\mathcal{O})$
are denoted respectively by $((\cdot,\cdot))$ and $\|\cdot\|$. We endow the space $\mathbb{V}$ with the norm generated
by the following scalar product
$$
(u,v)_\mathbb{V}=(u,v)+\alpha ((u,v)),\ \text{for any } v\in\mathbb{V};
$$
which is equivalent to $\|\cdot\|$, more precisely, we have
\begin{eqnarray*}
(\mathcal{P}^2+\alpha)^{-1}\|v\|^2_\mathbb{V}
\leq
\|v\|^2
\leq
\alpha^{-1}\|v\|^2_\mathbb{V},\ \ for\ any\ v\in\mathbb{V},
\end{eqnarray*}
where $\mathcal{P}$ is the constant from Poincar\'e's inequality.

We also introduce the following space
\begin{eqnarray*}
\mathbb{W}=\{u\in\mathbb{V}\ \text{such that }curl (u-\alpha\triangle u)\in L^2(\mathcal{O})\},
\end{eqnarray*}
and endow it with the norm generated by the scalar product
\begin{eqnarray}\label{W}
(u,v)_\mathbb{W}=(u,v)_\mathbb{V}+\Big(curl(u-\alpha\triangle u),curl(v-\alpha\triangle v)\Big).
\end{eqnarray}
The following result states that $(\cdot,\cdot)_\mathbb{W}$ is equivalent to the usual $\mathbb{H}^3(\mathcal{O})$-norm on $\mathbb{W}$, and can be found
in \cite{CE} \cite{CG} and Lemma 2.1 in \cite{RS-12}.

\begin{lem}
Set
$
\widetilde{\mathbb{W}}=\Big\{v\in\mathbb{H}^3(\mathcal{O})\text{ such that }{\rm div} v=0\ and\ v|_{\partial \mathcal{O}}=0\Big\}.
$
Then the following (algebraic and topological) identity holds:
\begin{eqnarray}\label{W = W}
\mathbb{W}=\widetilde{\mathbb{W}}.
\end{eqnarray}

Moreover, there is a positive constant $C$ such that
\begin{eqnarray}\label{W-02}
    \|v\|^2_{\mathbb{H}^3(\mathcal{O})}
\leq
    C\Big(\|v\|^2_\mathbb{V}+|curl(v-\alpha \triangle v)|^2\Big),
\end{eqnarray}
for any $v\in \widetilde{\mathbb{W}}$.
\end{lem}


From now on, we identify the space $\mathbb{V}$ with its dual space $\mathbb{V}^*$ via the Riesz representation, and we have the
Gelfand triple
\begin{eqnarray}\label{Gelfand}
\mathbb{W}\subset \mathbb{V}\subset\mathbb{W}^*.
\end{eqnarray}
 We denote by $\langle f,v\rangle$ the
action of the element $f$ of $\mathbb{W}^*$ on an element $v\in\mathbb{W}$. It is easy to see
$$(v,w)_\mathbb{V}=\langle v,w\rangle,\ \ \ \forall v\in\mathbb{V},\ \ \forall w\in\mathbb{W}.$$

Note that the injection of $\mathbb{W}$ into $\mathbb{V}$ is compact. Thus,
there exists a sequence $\{e_i:i=1,2,3,\cdots\}$ of elements of $\mathbb{W}$ which forms an orthonormal basis in $\mathbb{W}$,
and an orthogonal basis in $\mathbb{V}$.
The elements of this sequence are the solutions of the eigenvalue problem
\begin{eqnarray}\label{Basis}
(v,e_i)_{\mathbb{W}}=\lambda_i(v,e_i)_{\mathbb{V}},\ \text{for any }v\in\mathbb{W}.
\end{eqnarray}
Here $\{\lambda_i:i=1,2,3,\cdots\}$ is an increasing sequence of positive eigenvalues. We have the following important result from
\cite{CG} about the regularity of the functions $e_i,\ i=1,2,3,\cdots.$

\begin{lem}\label{lem Basis}
Let $\mathcal{O}$ be a bounded, simply-connected open subset of $\mathbb{R}^2$ with a boundary of class $\mathcal{C}^3$, then
the eigenfunctions of (\ref{Basis}) belong to $\mathbb{H}^4(\mathcal{O})$.
\end{lem}

 Consider the following ``generalized Stokes equations":

\begin{eqnarray}\label{General Stokes}
v-\alpha \triangle v+\nabla q=f\ {\rm in}\ \mathcal{O},\nonumber\\
{\rm div}\ v=0\ {\rm in}\ \mathcal{O},\\
v=0\ {\rm on}\ \partial \mathcal{O}.\nonumber
\end{eqnarray}

The following result can be derived from \cite{SO1}, \cite{SO2} and also can be found in \cite{RS-10} and \cite{RS-12}.
\begin{lem}\label{Lem GS}
Let $\mathcal{O}$ be a connected, bounded open subset of $\mathbb{R}^2$ with boundary $\partial \mathcal{O}$ of class $\mathcal{C}^l$
and let $f$ be a function in $\mathbb{H}^l$, $l\geq 1$. Then the system (\ref{General Stokes}) admits a solution $v\in \mathbb{H}^{l+2}\cap\mathbb{V}$.
Moreover if $f$ is an element of $\mathbb{H}$, then $v$ is unique and the following relations hold
\begin{eqnarray}\label{Eq GS-01}
(v,g)_\mathbb{V}=(f,g),\ \forall g\in \mathbb{V},
\end{eqnarray}
\begin{eqnarray}\label{Eq GS-02}
\|v\|_{\mathbb{H}^{l+2}}\leq C\|f\|_\mathbb{H}.
\end{eqnarray}
\end{lem}

Define the Stokes operator by
\begin{eqnarray}\label{Eq Stoke}
Au=-\mathbb{P}\triangle u,\ \forall u\in D(A)=\mathbb{H}^2(\mathcal{O})\cap\mathbb{V},
\end{eqnarray}
here we denote by $\mathbb{P}:\mathbb{L}^2(\mathcal{O})\rightarrow\mathbb{H}$ the usual Helmholtz-Leray projector.
It follows from Lemma \ref{Lem GS} that the operator $(I+\alpha A)^{-1}$ defines an isomorphism form $\mathbb{H}^l(\mathcal{O})\cap\mathbb{H}$
into $\mathbb{H}^{l+2}(\mathcal{O})\cap\mathbb{V}$ provided that $\mathcal{O}$ is of class $\mathcal{C}^l$, $l\geq1$. Moreover, the following
properties hold
\begin{eqnarray*}
& & ((I+\alpha A)^{-1}f,v)_\mathbb{V}=(f,v),\\
& & \|(I+\alpha A)^{-1}f\|_\mathbb{V}\leq C|f|,
\end{eqnarray*}
for any $f\in \mathbb{H}^l(\mathcal{O})\cap\mathbb{V}$ and any $v\in\mathbb{V}$. From these facts, $\widehat{A}=(I+\alpha A)^{-1}A$ defines a
continuous linear operator from $\mathbb{H}^l(\mathcal{O})\cap\mathbb{V}$ onto itself for $l\geq2$, and satisfies
$$
(\widehat{A}u,v)_\mathbb{V}=(Au,v)=((u,v)),
$$
for any $u\in\mathbb{W}$ and $v\in\mathbb{V}$. Hence
$$
(\widehat{A}u,u)_\mathbb{V}=\|u\|,
$$
for any $u\in\mathbb{W}$.

Let
$$
b(u,v,w)=\sum_{i,j=1}^2\int_{\mathcal{O}}u_i\frac{\partial v_j}{\partial x_i}w_jdx,
$$
for any $u,v,w\in\mathcal{C}$. Then the following identity holds(see for instance \cite{Bernard} \cite{CO}):
\begin{eqnarray}\label{curl}
((curl\Phi)\times v,w)=b(v,\Phi,w)-b(w,\Phi,v),
\end{eqnarray}
for any smooth function $\Phi,\ v$ and $w$. Now we recall the following two lemmas which can be found in \cite{RS-12}(Lemma 2.3 and Lemma 2.4),
 and also in  \cite{Bernard} \cite{CO}.

\begin{lem}\label{Lem B}
For any $u,v,w\in\mathbb{W}$, we have
\begin{eqnarray}\label{Ineq B 01}
    |(curl(u-\alpha\Delta u)\times v,w)|
\leq
    C\|u\|_{\mathbb{H}^3}\|v\|_\mathbb{V}\|w\|_{\mathbb{W}},
\end{eqnarray}
and
\begin{eqnarray}\label{Ineq B 02}
    |(curl(u-\alpha\Delta u)\times u,w)|
\leq
    C\|u\|^2_\mathbb{V}\|w\|_{\mathbb{W}}.
\end{eqnarray}
\end{lem}

Define the bilinear operator $\widehat{B}(\cdot,\cdot):\ \mathbb{W}\times\mathbb{V}\rightarrow\mathbb{W}^*$ as
\begin{eqnarray}\label{Lem B-01}
\widehat{B}(u,v)=(I+\alpha A)^{-1}\Big( curl(u-\alpha \Delta u)\times v\Big).
\end{eqnarray}
\begin{lem}\label{Lem-B-01}
For any $u\in\mathbb{W}$ and $v\in\mathbb{V}$ there holds
\begin{eqnarray}\label{Eq B-01}
    \|\widehat{B}(u,v)\|_{\mathbb{W}^*}
\leq
    C\|u\|_\mathbb{W}\|v\|_\mathbb{V},
\end{eqnarray}
and
\begin{eqnarray}\label{Eq B-02}
 \|\widehat{B}(u,u)\|_{\mathbb{W}^*}
\leq
    C_B\|u\|^2_\mathbb{V}.
\end{eqnarray}
In addition
\begin{eqnarray}\label{Eq B-03}
 \langle\widehat{B}(u,v),v\rangle=0,
\end{eqnarray}
which implies
\begin{eqnarray}\label{Eq B-04}
 \langle\widehat{B}(u,v),w\rangle=-\langle\widehat{B}(u,w),v\rangle,
\end{eqnarray}
for any $u,\ v,\ w\in\mathbb{W}$.

\end{lem}

\section{Hypotheses}

In this section, we will state the precise assumptions on the coefficients and collect some preliminary results from \cite{RS-10} and \cite{RS-12}, which will be used in the later sections.

We endow the complete probability space $(\Omega,\mathcal{F},P)$ with the filtration $\mathcal{F}_t$, $t\in[0,T]$. Let $F:\mathbb{V}\times[0,T]\rightarrow\mathbb{V}$
and $G:\mathbb{V}\times[0,T]\rightarrow\mathbb{V}^{\otimes m}$ be given measurable maps. We introduce the following conditions:

{\bf (F)} For any $t\in[0,T]$ and for any $u_1,u_2\in\mathbb{V}$,
\begin{eqnarray}\label{F-01}
            F(0,t)=0,
         \end{eqnarray}
and
\begin{eqnarray}\label{F-02}
\|F(u_1,t)-F(u_2,t)\|_\mathbb{V}
\leq
C\|u_1-u_2\|_\mathbb{V}.
\end{eqnarray}

{\bf (G)} For any $t\in[0,T]$ and for any $u_1,u_2\in\mathbb{V}$,
\begin{eqnarray}\label{G-01}
G(0,t)=0,
\end{eqnarray}
and
\begin{eqnarray}\label{G-02}
\|G(u_1,t)-G(u_2,t)\|_{\mathbb{V}^{\otimes m}}
\leq
C\|u_1-u_2\|_\mathbb{V}.
\end{eqnarray}

We now define  two operators $\widehat{F}$ and $\widehat{G}$ which map $\mathbb{V}\times[0,T]$ into $\mathbb{W}$ and
$\mathbb{W}^{\otimes m}$, respectively, by
\begin{eqnarray*}
\widehat{F}(u,t)=(I+\alpha A)^{-1} F(u,t),\ \ \ \widehat{G}(u,t)=(I+\alpha A)^{-1} G(u,t).
\end{eqnarray*}

{\bf Condition (F)} and {\bf Condition (G)}implies that there exist $C_F$, $C_G$
such that
\begin{eqnarray}\label{Jian F}
\|\widehat{F}(u_1,t)-\widehat{F}(u_2,t)\|_\mathbb{V}\leq C_F\|u_1-u_2\|_\mathbb{V},
\end{eqnarray}
\begin{eqnarray}\label{Jian G}
\|\widehat{G}(u_1,t)-\widehat{G}(u_2,t)\|_{\mathbb{V}^{\otimes m}}\leq C_G\|u_1-u_2\|_\mathbb{V}.
\end{eqnarray}

Alongside (\ref{01}), we consider the abstract stochastic evolution equations
\begin{eqnarray}\label{Abstract}
du^\epsilon(t)+\nu \widehat{A}u^\epsilon(t)dt+\widehat{B}(u^\epsilon(t),u^\epsilon(t))dt=\widehat{F}(u^\epsilon(t),t)+\sqrt{\epsilon}\widehat{G}(u^\epsilon(t),t)dW(t),
\end{eqnarray}
with initial value $u_0=u(0)$, which holds in $\mathbb{W}^*$. It can be proved that a stochastic
process $u^\epsilon$ satisfies (\ref{Abstract}) if and only if it verifies (\ref{01}) in the weak sense of partial differential equations.
Indeed, (\ref{Abstract}) is obtained by applying $(I+\alpha A)^{-1}$ to the equation (\ref{01}).

Now we recall the concept of solution of the problem (\ref{01}) in \cite{RS-12}.

\begin{dfn}\label{Def 01}
A stochastic process $u^\epsilon$ is called a solution of the system (\ref{01}), if the following three conditions hold

1. $u^\epsilon\in L^p(\Omega,\mathcal{F},P;L^\infty([0,T],\mathbb{W})),\ 2\leq p<\infty.$

2. For all $t$, $u^\epsilon(t)$ is $\mathcal{F}_t$-measurable.

3. For any $t\in(0,T]$ and $v\in\mathbb{W}$, the following identity holds almost surely
\begin{eqnarray*}
&  &(u^\epsilon(t)-u^\epsilon(0),v)_{\mathbb{V}}+\int_0^t[\nu ((u^\epsilon(s),v))+(curl (u^\epsilon(s)-\alpha \Delta u^\epsilon(s))\times u^\epsilon(s),v)]ds\\
&=&
    \int_0^t(F(u^\epsilon(s),s),v)ds+\sqrt{\epsilon}\int_0^t(G(u^\epsilon(s),s),v)dW(s).
\end{eqnarray*}
 Or equivalently, for any $t\in(0,T]$, the following equation
\begin{eqnarray*}
u^\epsilon(t)+\int_0^t\Big(\nu \widehat{A}u^\epsilon(s)+\widehat{B}(u^\epsilon(s),u^\epsilon(s))\Big)ds
=
u_0+\int_0^t\widehat{F}(u^\epsilon(s),s)ds+\sqrt{\epsilon}\int_0^t\widehat{G}(u^\epsilon(s),s)dW(s),
\end{eqnarray*}
holds in $\mathbb{W}^*$ $P$-a.s..
\end{dfn}

Using Galerkin approximation scheme for the system (\ref{01}), Razafimandimby and Sango \cite{RS-12}
obtained the following theorem (see Theorem 3.4 and Theorem 4.1 in \cite{RS-12}).
\begin{thm}\label{Solution Existence}
Let $u_0\in\mathbb{W}$. Assume conditions {\bf (F)} and {\bf (G)} hold. Then the system (\ref{01}) or the problem
(\ref{Abstract}) has a unique solution. Moreover,
the solution $u^\epsilon$ admits a version which is continuous in $\mathbb{V}$ with respect to
the strong topology and continuous in $\mathbb{W}$ with respect to
the weak topology.
\end{thm}

\section{Large Deviation Principle}

In this section, we will establish a large deviation principle for system (\ref{01}).   We first recall the general criteria obtained in \cite{Budhiraja-Dupuis}.
\vskip0.3cm
Let $(\Omega,\mathcal{F},\mathbb{P})$ be a probability space with an increasing family $\{\FF_t\}_{0\le t\le T}$ of the sub-$\sigma$-fields of $\FF$ satisfying the usual conditions.
Let $\mathcal{E}$ be a Polish space with the Borel $\sigma$-field $\mathcal{B}(\mathcal{E})$.
    \begin{dfn}\label{Dfn-Rate function}
       \emph{\textbf{(Rate function)}} A function $I: \mathcal{E}\rightarrow[0,\infty]$ is called a rate function on
       $\mathcal{E}$,
       if for each $M<\infty$, the level set $\{x\in\mathcal{E}:I(x)\leq M\}$ is a compact subset of $\mathcal{E}$.
         \end{dfn}
    \begin{dfn}
       \emph{\textbf{(Large deviation principle)}} Let $I$ be a rate function on $\mathcal{E}$.  A family
       $\{X^\e\}$ of $\EE$-valued random elements is  said to satisfy the large deviation principle on $\mathcal{E}$
       with rate function $I$, if the following two conditions
       hold.
       \begin{itemize}
         \item[$(a)$](Upper bound) For each closed subset $F$ of $\mathcal{E}$,
              $$
                \limsup_{\e\rightarrow 0}\e\log\mathbb{P}(X^\e\in F)\leq- \inf_{x\in F}I(x).
              $$
         \item[$(b)$](Lower bound) For each open subset $G$ of $\mathcal{E}$,
              $$
                \liminf_{\e\rightarrow 0}\e\log\mathbb{P}(X^\e\in G)\geq- \inf_{x\in G}I(x).
              $$
       \end{itemize}
    \end{dfn}

\vskip0.3cm

The Cameron-Martin space associated with the Wiener process $\{W(t), t\in[0,T]\}$ is given by
\beq\label{Cameron-Martin}
\HH_0:=\left\{h:[0,T]\rightarrow \mathbb{R}^m; h \  \text{is absolutely continuous and } \int_0^T\|\dot h(s)\|_{\mathbb{R}^m}^2ds<+\infty\right\}.
\nneq
The space $\HH_0$ is a Hilbert space with inner product
 $$
 \langle h_1, h_2\rangle_{\HH_0}:=\int_0^T\langle \dot h_1(s), \dot h_2(s)\rangle_{\mathbb{R}^m}ds.
 $$

Let $\AA$ denote  the class of $\mathbb{R}^m$-valued $\{\FF_t\}$-predictable processes $\phi$ belonging to $\HH_0$ a.s..
Let $S_N=\{h\in \HH_0; \int_0^T\|\dot h(s)\|_{\mathbb{R}^m}^2ds\le N\}$. The set $S_N$ endowed with the weak topology is a Polish space.
Define $\AA_N=\{\phi\in \AA;\phi(\omega)\in S_N, \mathbb{P}\text{-a.s.}\}$.

\vskip0.3cm

 Recall the following result from Budhiraja and Dupuis \cite{Budhiraja-Dupuis}.

\bthm\label{thm BD}{\rm(\cite{Budhiraja-Dupuis}) For $\e>0$, let $\Gamma^\e$ be a measurable mapping from $C([0,T];\mathbb{R}^m)$ into $\EE$.
Let $X^\e:=\Gamma^\e(W(\cdot))$. Suppose that
there  exists a measurable map $\Gamma^0:C([0,T];\mathbb{R}^m)\rightarrow \EE$ such that
\begin{itemize}
   \item[(a)] for every $N<+\infty$ and any family $\{h^\e;\e>0\}\subset \AA_N$ satisfying that $h^\e$ converge in distribution as $S_N$-valued random elements to $h$ as $\e\rightarrow 0$,
    $\Gamma^\e\left(W(\cdot)+\frac{1}{\sqrt\e}\int_0^{\cdot}\dot h^\e(s)ds\right)$ converges in distribution to $\Gamma^0(\int_0^{\cdot}\dot h(s)ds)$ as $\e\rightarrow 0$;
   \item[(b)] for every $N<+\infty$, the set
   $$
 \left\{\Gamma^0\left(\int_0^{\cdot}\dot h(s)ds\right); h\in S_N\right\}
  $$
   is a compact subset of $\EE$.
 \end{itemize}
Then the family $\{X^\e\}_{\e>0}$ satisfies a large deviation principle in $\EE$ with the rate function $I$ given by
\beq\label{rate function}
I(g):=\inf_{\{h\in \HH_0;g=\Gamma^0(\int_0^{\cdot}\dot h(s)ds)\}}\left\{\frac12\int_0^T\|\dot h(s)\|_{\mathbb{R}^m}^2ds\right\},\ g\in\EE,
\nneq
with the convention $\inf\{\emptyset\}=\infty$.
 }\nthm

\subsection{Main Results}
The strong solutions of equation (\ref{01}) determine a measurable mapping $\Gamma^\e(\cdot)$ from $C([0,T];\mathbb{R}^m)$
into $C([0,T],\mathbb{V})$ so that $\Gamma^\e(W)=u^\e$.

Let $N$ be any fixed positive number.
Fixed $g\in S_N$, consider the following deterministic PDE:
\begin{eqnarray}\label{Eq Skep}
& &u^{g}(t)+\int_0^t\Big(\nu \widehat{A}u^{g}(s)+\widehat{B}(u^{g}(s),u^{g}(s))\Big)ds\nonumber\\
&=&
u_0+\int_0^t\widehat{F}(u^{g}(s),s)ds+\int_0^t\widehat{G}(u^{g}(s),s)\dot{g}(s)ds.
\end{eqnarray}
For any family $\{h^\e;\e>0\}\subset \AA_N$, let $u^{h^\e}$ be the solution of the following SPDE
\begin{eqnarray}\label{Eq LDP1}
&  &u^{h^\e}(t)+\int_0^t\Big(\nu \widehat{A}u^{h^\e}(s)+\widehat{B}(u^{h^\e}(s),u^{h^\e}(s))\Big)ds\\
&=&
u_0+\int_0^t\widehat{F}(u^{h^\e}(s),s)ds+\sqrt{\epsilon}\int_0^t\widehat{G}(u^{h^\e}(s),s)dW(s)
+
\int_0^t\widehat{G}(u^{h^\e}(s),s)\dot{h^\e}(s) ds.\nonumber
\end{eqnarray}

Then it is easy to see that $\Gamma^\e\left(W(\cdot)+\frac{1}{\sqrt\e}\int_0^{\cdot}\dot h^\e(s)ds\right)=u^{h^\e}$. Define $\Gamma^0(\int_0^{\cdot}\dot g(s)ds)=u^{g}$.
Let $I:C([0,T],\mathbb{V})\rightarrow [0,\infty]$ be defined as in (\ref{rate function}).

\bthm\label{th main}
Assume that the Lipschitz conditions {\bf (F)} and {\bf (G)} hold. Then the solution family $\{u^\epsilon\}_{\epsilon>0}$ of
system (\ref{01}) satisfies
 a large deviation principle on $C([0,T],\mathbb{V})$ with the good rate function $I$ with respect to the topology of uniform convergence.

\nthm

\noindent{\bf Proof of Theorem \ref{th main}.}

According to Theorem \ref{thm BD}, we need to prove that Condition (a), (b) are fulfilled. The verification of Condition (a) will be given
 by Theorem \ref{Thm condition 02} below. Condition (b) will be established in Theorem \ref{Thm condition 01} below.

%

\subsection{Proof of Theorem \ref{th main}}\label{Sub01}

From now on, we denote by $C$ any generic constant which may change from one line to another.


First we will establish the following a priori estimate.
\begin{lem}\label{Lemma Esta1}
There exists $\epsilon_0>0$ such that
\begin{eqnarray}\label{Estation-02}
\sup_{\epsilon\in(0,\epsilon_0)}E(\sup_{s\in[0,T]}\|u^{h^\e}(s)\|^p_\mathbb{W})\leq C_{p,N},\ \text{for any }2\leq p<\infty,
\end{eqnarray}
here $C_{p,N}$ is independent of $\e$.
\end{lem}

Set $\mathbb{W}_M=Span(e_1,\cdots,e_M)$. Let $u^M\in\mathbb{W}_M$ be the Galerkin approximations of (\ref{Eq LDP1}) satisfying
\begin{eqnarray}\label{Eq Galerkin}
& &d(u^M,e_i)_\mathbb{V}+\nu((u^M,e_i))dt+b(u^M,u^M,e_i)dt-\alpha b(u^M,\triangle u^M,e_i)dt+\alpha b(e_i,\triangle u^M,u^M)dt\nonumber\\
&=&
 (F(u^M,t),e_i)dt+\sqrt{\e}(G(u^M,t),e_i)dW(t)+(G(u^M,t)\dot{h^\e}(t),e_i)dt
\end{eqnarray}
for any $i\in\{1,2,\cdots,M\}$.

As in the proof of Theorem 3.4 in \cite{RS-12}, one can show that $u^M\rightarrow u^{h^\e}$ weakly-* in $L^p(\Omega,\mathcal{F},P,L^\infty([0,T],\mathbb{W}))$
for any $p\geq 2$. Hence Lemma \ref{Lemma Esta1} will follow from the following Lemma \ref{Lem G 02}.
\begin{lem}\label{Lem G 02}
For any $4\leq p<\infty$,  we have, for any $\e\in(0,1)$
\begin{eqnarray}\label{Eq G02 01}
E\Big(\sup_{s\in[0,T]}\|u^M(s)\|^p_\mathbb{V}\Big)\leq C_{p,N},
\end{eqnarray}
and
\begin{eqnarray}\label{Eq G02 02}
E\Big(\sup_{s\in[0,T]}\|u^M(s)\|^p_\mathbb{W}\Big)\leq C_{p,N}.
\end{eqnarray}

\end{lem}

\begin{proof}
Set $\|v\|_*=|curl(v-\alpha\triangle v)|$ for any $v\in\mathbb{W}$. Define
$$\tau_J=\inf\{t\geq0,\|u^M(s)\|_\mathbb{V}+\|u^M(s)\|_*\geq J\}.$$

Applying ${\rm It\hat{o}}$'s formula, we have
\begin{eqnarray*}
& &d(u^M,e_i)^2_{\mathbb{V}}\nonumber\\
& &+
2(u^M,e_i)_{\mathbb{V}}\Big[\nu((u^M,e_i))+b(u^M,u^M,e_i)-\alpha b(u^M,\triangle u^M,e_i)+\alpha b(e_i,\triangle u^M,u^M)\Big]dt\nonumber\\
&=&
2(u^M,e_i)_{\mathbb{V}}\Big[(F(u^M,t),e_i)dt+\sqrt{\e}(G(u^M,t),e_i)dW(t)+(G(u^M,t)\dot{h^\e}(t),e_i)dt\Big]\nonumber\\
& &+
\e(G(u^M,t),e_i)(G(u^M,t),e_i)' dt.
\end{eqnarray*}
Noting that $\|u^M\|^2_\mathbb{V}=\sum_{i=1}^M\lambda_i(u^M,e_i)^2_\mathbb{V}$,
\begin{eqnarray}\label{Eq G ito}
&  &d\|u^M\|^2_\mathbb{V}+2\nu\|u^M\|^2dt\nonumber\\
&=&
2(F(u^M,t),u^M)dt+2\sqrt{\e}(G(u^M,t),u^M)dW(t)+2(G(u^M,t)\dot{h^\e}(t),u^M)dt\nonumber\\
&  &+
\e\sum_{i=1}^M\lambda_i(G(u^M,t),e_i)(G(u^M,t),e_i)' dt,
\end{eqnarray}
here we have used the fact that $b(u^M,u^M,u^M)=0$. Applying ${\rm It\hat{o}}$'s formula to $\|u^M\|^p_\mathbb{V}$, we have
\begin{eqnarray}\label{Eq G 00}
d\|u^M\|^p_\mathbb{V}
&=&d(\|u^M\|^2_\mathbb{V})^{p/2}=
\frac{p}{2}(\|u^M\|^2_\mathbb{V})^{\frac{p}{2}-1}d\|u^M\|^2_\mathbb{V}
  +
\frac{p}{4}(\frac{p}{2}-1)(\|u^M\|^2_\mathbb{V})^{\frac{p}{2}-2}d\langle\|u^M\|^2_\mathbb{V}\rangle\nonumber\\
&=&
\frac{p}{2}\|u^M\|_\mathbb{V}^{p-2}
    \Big(
         -2\nu\|u^M\|^2dt+2(F(u^M,t),u^M)dt+2\sqrt{\e}(G(u^M,t),u^M)dW(t)\nonumber\\
         & &+2(G(u^M,t)\dot{h^\e}(t),u^M)dt
          +
         \e\sum_{i=1}^M\lambda_i(G(u^M,t),e_i)(G(u^M,t),e_i)' dt
    \Big)\nonumber\\
     & &+
     \e p(\frac{p}{2}-1)\|u^M\|_\mathbb{V}^{p-4}(G(u^M,t),u^M)(G(u^M,t),u^M)' dt.
\end{eqnarray}

Recall that $\mathcal{P}$ is the ${\rm Poincar\acute{e}}$'s constant. We have
\begin{eqnarray}\label{Eq ito 02}
|(F(u^M(s),s),u^M(s))|
\leq
C\mathcal{P}^2\|u^M(s)\|^2
\leq
\frac{C\mathcal{P}^2}{\alpha}\|u^M(s)\|^2_\mathbb{V},
\end{eqnarray}
and
\begin{eqnarray}\label{Eq G 02}
|(G(u^M,t),u^M)(G(u^M,t),u^M)'|
\leq
C\|u^M\|^4_\mathbb{V}.
\end{eqnarray}

By  Burkholder-Davis-Gundy inequalities,
\begin{eqnarray}\label{Eq G 03}
&  &E\Big(\sup_{t\in[0,T]}|\int_0^{t\wedge\tau_J}\|u^M(s)\|_\mathbb{V}^{p-2}(G(u^M(s),s),u^M(s))dW(s)|\Big)\nonumber\\
&\leq&
CE\Big(\int_0^{T\wedge\tau_J}\|u^M(s)\|_\mathbb{V}^{2p}ds\Big)^{1/2}\nonumber\\
&\leq&
CE\Big[\sup_{t\in[0,{T\wedge\tau_J}]}\|u^M(t)\|_\mathbb{V}^{p/2}\Big(\int_0^{T\wedge\tau_J}\|u^M(s)\|_\mathbb{V}^{p}ds\Big)^{1/2}\Big]\nonumber\\
&\leq&
\delta E\Big(\sup_{t\in[0,{T\wedge\tau_J}]}\|u^M(t)\|_\mathbb{V}^{p}\Big)+C_\delta E\Big(\int_0^{T\wedge\tau_J}\|u^M(s)\|_\mathbb{V}^{p}ds\Big).
\end{eqnarray}

By ${\rm H\ddot{o}lder}$'s inequality and Young's inequality, for any $\eta>0$
\begin{eqnarray}\label{Eq G 04}
& &\int_0^{t\wedge\tau_J}\|u^M(s)\|_\mathbb{V}^{p-2}|(G(u^M(s),s)\dot{h^\e}(s),u^M(s))|ds\nonumber\\
&\leq&
C\int_0^{t\wedge\tau_J}\|u^M(s)\|_\mathbb{V}^{p}\|\dot{h^\e}(s)\|_{\mathbb{R}^m}ds\nonumber\\
&\leq&
C\sup_{s\in[0,{T\wedge\tau_J}]}\|u^M(s)\|_\mathbb{V}^{p-1}\int_0^{t\wedge\tau_J}\|u^M(s)\|_\mathbb{V}\|\dot{h^\e}(s)\|_{\mathbb{R}^m}ds\nonumber\\
&\leq&
\eta \sup_{s\in[0,{T\wedge\tau_J}]}\|u^M(s)\|_\mathbb{V}^{p}+C_\eta(\int_0^{t\wedge\tau_J}\|u^M(s)\|_\mathbb{V}\|\dot{h^\e}(s)\|_{\mathbb{R}^m}ds)^p\nonumber\\
&\leq&
\eta \sup_{s\in[0,{T\wedge\tau_J}]}\|u^M(s)\|_\mathbb{V}^{p}+C_\eta(\int_0^{t\wedge\tau_J}\|u^M(s)\|^2_\mathbb{V}ds\int_0^{t\wedge\tau_J}\|\dot{h^\e}(s)\|^2_{\mathbb{R}^m}ds)^{p/2}\nonumber\\
&\leq&
\eta \sup_{s\in[0,{T\wedge\tau_J}]}\|u^M(s)\|_\mathbb{V}^{p}+C_\eta N^{p/2}(\int_0^{t\wedge\tau_J}\|u^M(s)\|^2_\mathbb{V}ds)^{p/2}\nonumber\\
&\leq&
\eta \sup_{s\in[0,{T\wedge\tau_J}]}\|u^M(s)\|_\mathbb{V}^{p}+C_\eta N^{p/2}T^{\frac{p-2}{2}}\int_0^{t\wedge\tau_J}\|u^M(s)\|^p_\mathbb{V}ds.
\end{eqnarray}

By Lemma \ref{Lem GS}, there exists unique solution $\widetilde{G}(u^M,t)\in\mathbb{W}^{\otimes m}$ satisfying
\begin{eqnarray*}
\widetilde{G}-\alpha \triangle \widetilde{G}=G(u^M,t)\ {\rm in}\ \mathcal{O},\nonumber\\
{\rm div}\ \widetilde{G}=0\ {\rm in}\ \mathcal{O},\\
\widetilde{G}=0\ {\rm on}\ \partial \mathcal{O}.\nonumber
\end{eqnarray*}
Moreover,
$$
(\widetilde{G}(u^M,t),e_i)_\mathbb{V}=(G(u^M,t),e_i),\ \ \forall i\in\{1,2,\cdots,M\},
$$
and there exists a positive constant $C_0$ such that
$$
\|\widetilde{G}(u^M,t)\|_{\mathbb{W}^{\otimes m}}
\leq
C_0\|G(u^M,t)\|_{\mathbb{V}^{\otimes m}}.
$$
Hence by (\ref{Basis}),
\begin{eqnarray}\label{Eq ito 04}
& &\sum_{i=1}^M\lambda_i(G(u^M(s),s),e_i)(G(u^M(s),s),e_i)'\nonumber\\
&=&
\sum_{i=1}^M\lambda_i(\widetilde{G}(u^M(s),s),e_i)_\mathbb{V}(\widetilde{G}(u^M(s),s),e_i)_\mathbb{V}'\nonumber\\
&=&
\sum_{i=1}^M\frac{1}{\lambda_i}(\widetilde{G}(u^M(s),s),e_i)_\mathbb{W}(\widetilde{G}(u^M(s),s),e_i)_\mathbb{W}'\nonumber\\
&\leq&
\frac{1}{\lambda_1}\|\widetilde{G}(u^M(s),s)\|^2_{\mathbb{W}^{\otimes m}}\nonumber\\
&\leq&
\frac{C_0}{\lambda_1}\|G(u^M(s),s)\|^2_{\mathbb{V}^{\otimes m}}\nonumber\\
&\leq&
C\|u^M(s)\|^2_\mathbb{V}.
\end{eqnarray}

Combining (\ref{Eq G 00})--(\ref{Eq ito 04}), we have
\begin{eqnarray}\label{Eq G 05}
& &(1-p\eta-\sqrt{\e}p\delta)E\Big(\sup_{t\in[0,{T\wedge\tau_J}]}\|u^M(t)\|^p_\mathbb{V}\Big)+2\nu E\int_0^{T\wedge\tau_J}\frac{p}{2}\|u^M\|_\mathbb{V}^{p-2}\|u^M\|^2ds\nonumber\\
&\leq&\|u(0)\|^p_\mathbb{V}
    +C_{\eta,N, p,\delta}E\int_0^{T\wedge\tau_J}\|u^M(s)\|^p_\mathbb{V}ds.
\end{eqnarray}
Choosing $\eta=\delta=\frac{1}{4p}$, then for any $\e\in(0,1)$
\begin{eqnarray}\label{Eq G 06}
E\Big(\sup_{t\in[0,{T\wedge\tau_J}]}\|u^M(t)\|^p_\mathbb{V}\Big)
\leq
    C_{p,N}.
\end{eqnarray}
Let $J\rightarrow\infty$ to obtain (\ref{Eq G02 01}).\\

Now we prove (\ref{Eq G02 02}).

Setting
$$
\phi(u^M)=-\nu\triangle u^M+curl(u^M-\alpha\triangle u^M)\times u^M-F(u^M,t)-G(u^M,t)\dot{h^\e}(t),
$$
we have
$$
d(u^M,e_i)_\mathbb{V}+(\phi(u^M),e_i)dt=\sqrt{\e}(G(u^M,t),e_i)dW(t).
$$

Note that $\phi(u^M)\in\mathbb{H}^1(\mathcal{O})$. By Lemma \ref{Lem GS}, there exists a unique solution $v^M\in\mathbb{W}$ satisfying
\begin{eqnarray*}
v^M-\alpha \triangle v^M=\phi(u^M)\ {\rm in}\ \mathcal{O},\nonumber\\
{\rm div}\ v^M=0\ {\rm in}\ \mathcal{O},\\
v^M=0\ {\rm on}\ \partial \mathcal{O}.\nonumber
\end{eqnarray*}
Moreover,
$$
(v^M,e_i)_\mathbb{V}=(\phi(u^M),e_i),\ \ \forall i\in\{1,2,\cdots,M\}.
$$
Thus
$$
d(u^M,e_i)_\mathbb{V}+(v^M,e_i)_\mathbb{V}dt=\sqrt{\e}(G(u^M,t),e_i)dW(t).
$$

We introduce $\widetilde{G}$ as in the proof of (\ref{Eq ito 04}) to get
$$
\lambda_i(G(u^M,t),e_i)=(\widetilde{G}(u^M,t),e_i)_{\mathbb{W}}.
$$
By (\ref{Basis}),
$$
d(u^M,e_i)_\mathbb{W}+(v^M,e_i)_\mathbb{W}dt=\sqrt{\e}(\widetilde{G}(u^M,t),e_i)_{\mathbb{W}}dW(t).
$$
Applying ${\rm It\hat{o}}$'s formula, we have
\begin{eqnarray*}
&  &d(u^M,e_i)^2_\mathbb{W}+2(u^M,e_i)_\mathbb{W}(v^M,e_i)_\mathbb{W}dt\\
&=&
2\sqrt{\e}(u^M,e_i)_\mathbb{W}(\widetilde{G}(u^M,t),e_i)_{\mathbb{W}}dW(t)
+
\e(\widetilde{G}(u^M,t),e_i)_{\mathbb{W}}(\widetilde{G}(u^M,t),e_i)_{\mathbb{W}}'dt,
\end{eqnarray*}
and
\begin{eqnarray*}
&  &d\|u^M\|^2_\mathbb{W}+2(v^M,u^M)_\mathbb{W}dt\\
&=&
2\sqrt{\e}(\widetilde{G}(u^M,t),u^M)_{\mathbb{W}}dW(t)
+
\e\sum_{i=1}^M(\widetilde{G}(u^M,t),e_i)_{\mathbb{W}}(\widetilde{G}(u^M,t),e_i)_{\mathbb{W}}'dt.
\end{eqnarray*}
 By (\ref{W}) we rewrite the above
equation as follows
\begin{eqnarray*}
&  &d[\|u^M\|^2_\mathbb{V}+\|u^M\|^2_*]+2\Big[(v^M,u^M)_\mathbb{V}+\Big(curl(u^M-\alpha \triangle u^M), curl(v^M-\alpha \triangle v^M)\Big)\Big]dt\\
&=&
2\sqrt{\e}(\widetilde{G}(u^M,t),u^M)_{\mathbb{V}}dW(t)
+
\e\sum_{i=1}^M\lambda_i^2(\widetilde{G}(u^M,t),e_i)_{\mathbb{V}}(\widetilde{G}(u^M,t),e_i)_{\mathbb{V}}'dt\\
& &+
2\sqrt{\e}\Big(curl(u^M-\alpha \triangle u^M), curl(\widetilde{G}(u^M,t)-\alpha \triangle \widetilde{G}(u^M,t))\Big)dW(t).
\end{eqnarray*}
By the definition of $v^M$ and $\widetilde{G}$, we obtain
\begin{eqnarray*}
&  &d[\|u^M\|^2_\mathbb{V}+\|u^M\|^2_*]+2\Big[(\phi(u^M),u^M)+\Big(curl(u^M-\alpha \triangle u^M), curl(\phi(u^M))\Big)\Big]dt\\
&=&
2\sqrt{\e}(G(u^M,t),u^M)dW(t)
+
\e\sum_{i=1}^M\lambda_i^2(G(u^M,t),e_i)(G(u^M,t),e_i)'dt\\
& &+
2\sqrt{\e}\Big(curl(u^M-\alpha \triangle u^M), curl(G(u^M,t))\Big)dW(t).
\end{eqnarray*}
Subtracting (\ref{Eq G ito}) from the above equation, we obtain
\begin{eqnarray*}
&  &d\|u^M\|^2_*+2\Big(curl(u^M-\alpha \triangle u^M), curl(\phi(u^M))\Big)dt\\
&=&
\e\sum_{i=1}^M(\lambda_i^2-\lambda_i)(G(u^M,t),e_i)(G(u^M,t),e_i)'dt\\
&&+
2\sqrt{\e}\Big(curl(u^M-\alpha \triangle u^M), curl(G(u^M,t))\Big)dW(t).
\end{eqnarray*}

Since
$$
curl\Big(curl(u^M-\alpha\triangle u^M)\times u^M\Big)
=
(u^M\cdot \nabla)(curl(u^M-\alpha\triangle u^M)),
$$
we have
\begin{eqnarray*}
& &\Big(curl(u^M-\alpha \triangle u^M), curl(\phi(u^M))\Big)\\
&=&
 \Big(curl(u^M-\alpha \triangle u^M),curl(-\nu \triangle u^M)\Big)
-
 \Big(curl(u^M-\alpha \triangle u^M),curl(F(u^M,t)+G(u^M)\dot{h^\e}(t))\Big)\\
& &+
 \Big(curl(u^M-\alpha \triangle u^M),curl\Big(curl(u^M-\alpha\triangle u^M)\times u^M\Big)\Big)\\
&=&
\frac{\nu}{\alpha}\|u^M\|^2_*-\frac{\nu}{\alpha}\Big(curl(u^M-\alpha \triangle u^M),curl\ u^M\Big)\\
&&-
\Big(curl(u^M-\alpha \triangle u^M),curl(F(u^M,t)+G(u^M)\dot{h^\e}(t))\Big).
\end{eqnarray*}
Hence
\begin{eqnarray}\label{eq 000}
&  &d\|u^M\|^2_*+\frac{2\nu}{\alpha}\|u^M\|^2_*dt-\frac{2\nu}{\alpha}\Big(curl(u^M-\alpha \triangle u^M),curl\ u^M\Big)dt\nonumber\\
& &-2\Big(curl(u^M-\alpha \triangle u^M),curl(F(u^M,t)+G(u^M,t)\dot{h^\e}(t))\Big)dt\nonumber\\
&=&
\e\sum_{i=1}^M(\lambda_i^2-\lambda_i)(G(u^M,t),e_i)(G(u^M,t),e_i)'dt\nonumber\\
&&+
2\sqrt{\e}\Big(curl(u^M-\alpha \triangle u^M), curl(G(u^M,t))\Big)dW(t).
\end{eqnarray}
Applying ${\rm It\hat{o}}$'s formula, we have
\begin{eqnarray}\label{Eq 100}
d\|u^M\|^p_*
&=&d(\|u^M\|^2_*)^{p/2}
=\frac{p}{2}(\|u^M\|^2_*)^{p/2-1}d\|u^M\|^2_*+\frac{p}{4}(\frac{p}{2}-1)(\|u^M\|^2_*)^{p/2-2}d\langle\|u^M\|^2_*\rangle\nonumber\\
&=&
\frac{p}{2}\|u^M\|_*^{p-2}
\Big(-\frac{2\nu}{\alpha}\|u^M\|^2_*dt+\frac{2\nu}{\alpha}\Big(curl(u^M-\alpha \triangle u^M),curl\ u^M\Big)dt\nonumber\\
& &+ 2\Big(curl(u^M-\alpha \triangle u^M),curl(F(u^M,t)+G(u^M,t)\dot{h^\e}(t))\Big)dt\nonumber\\
& &+\e\sum_{i=1}^M(\lambda_i^2-\lambda_i)(G(u^M,t),e_i)(G(u^M,t),e_i)'dt\\
& &+
2\sqrt{\e}\Big(curl(u^M-\alpha \triangle u^M), curl(G(u^M,t))\Big)dW(t)
\Big)+
\e p(\frac{p}{2}-1)\|u^M\|_*^{p-4}\nonumber\\
& &\cdot\Big(curl (u^M-\alpha \triangle u^M),curl G(u^M,t)\Big)\Big(curl (u^M-\alpha \triangle u^M),curl G(u^M,t)\Big)'dt\nonumber.
\end{eqnarray}

Using the fact that
\begin{eqnarray}\label{Eq 9}
|curl(\phi)|^2\leq \frac{2}{\alpha}\|\phi\|^2_\mathbb{V}\ \ \ \text{for any }\phi\in\mathbb{W},
\end{eqnarray}
we have
\begin{eqnarray}\label{Eq 101}
& &\|u^M\|_*^{p-2}|(curl(u^M-\alpha \triangle u^M),curl\ u^M)|
+\|u^M\|_*^{p-2}|(curl(u^M-\alpha \triangle u^M),curl(F(u^M,t)))|\nonumber\\
&\leq&
C\|u^M\|_*^{p-1}\|u^M\|_\mathbb{V}\nonumber\\
&\leq&
\|u^M\|_*^{p}+C\|u^M\|^p_\mathbb{V},
\end{eqnarray}
and
\begin{eqnarray}\label{Eq 102-01}
& &\|u^M\|_*^{p-4}\Big(curl (u^M-\alpha \triangle u^M),curl\ G(u^M,t)\Big)\Big(curl (u^M-\alpha \triangle u^M),curl\ G(u^M,t)\Big)'\nonumber\\
&\leq&
C\|u^M\|_*^{p-2}\|u^M\|_\mathbb{V}^{2}\nonumber\\
&\leq&
\|u^M\|_*^{p}+C\|u^M\|^p_\mathbb{V}.
\end{eqnarray}

By the similar arguments as for the proof of (\ref{Eq ito 04}), we obtain
\begin{eqnarray*}
\sum_{i=1}^M(\lambda_i^2+\lambda_i)(G(u^M(s),s),e_i)(G(u^M(s),s),e_i)'
\leq
C\|u^M(s)\|^2_\mathbb{V}.
\end{eqnarray*}
Thus
\begin{eqnarray}\label{Eq 102-02}
& &\|u^M\|_*^{p-2}\sum_{i=1}^M(\lambda_i^2-\lambda_i)(G(u^M,t),e_i)(G(u^M,t),e_i)'\nonumber\\
&\leq&
C\|u^M\|_*^{p-2}\|u^M\|_\mathbb{V}^{2}\nonumber\\
&\leq&
\|u^M\|_*^{p}+C\|u^M\|^p_\mathbb{V}.
\end{eqnarray}

By ${\rm H\ddot{o}lder}$'s inequality and Young's inequality, for any $\eta>0$
\begin{eqnarray}\label{Eq 103}
& &\int_0^{t\wedge\tau_J}\|u^M\|_*^{p-2}|(curl(u^M-\alpha \triangle u^M),curl(G(u^M,s)\dot{h^\e}(s)))|ds\nonumber\\
&\leq&
C\int_0^{t\wedge\tau_J}\|u^M\|_*^{p-1}\|u^M\|_\mathbb{V}\|\dot{h^\e}(s))\|_{\mathbb{R}^m}ds\nonumber\\
&\leq&
\eta\sup_{s\in[0,{T\wedge\tau_J}]}\|u^M\|_*^p+C_{\eta,N,p}\int_0^{T\wedge\tau_J}\|u^M\|_\mathbb{V}^pds.
\end{eqnarray}

Applying Burkholder-Davis-Gundy inequalities,
\begin{eqnarray}\label{Eq 104}
& &E\Big(\sup_{t\in[0,T]}\Big|\int_0^{t\wedge\tau_J}\|u^M(s)\|_*^{p-2}\Big(curl(u^M(s)-\alpha \triangle u^M(s)), curl(G(u^M(s),s))\Big)dW(s)\Big|\Big)\nonumber\\
&\leq&
C E\Big(\int_0^{T\wedge\tau_J}\|u^M(s)\|_*^{2p-2}\|u^M(s)\|^2_\mathbb{V}ds\Big)^{1/2}\nonumber\\
&\leq&
CT^{1/2}E\Big(\sup_{t\in[0,{T\wedge\tau_J}]}\|u^M(s)\|_*^{p-1}\sup_{t\in[0,{T\wedge\tau_J}]}\|u^M(s)\|_\mathbb{V}\Big)\nonumber\\
&\leq&
\delta E\Big(\sup_{t\in[0,{T\wedge\tau_J}]}\|u^M(s)\|_*^{p}\Big)+C_\delta E\Big(\sup_{t\in[0,{T\wedge\tau_J}]}\|u^M(s)\|^p_\mathbb{V}\Big).
\end{eqnarray}

Combining (\ref{Eq 100})-(\ref{Eq 104}), for every $\e\in(0,1)$
\begin{eqnarray*}
& &(1-p\eta-\sqrt{\e}p\delta)E\Big(\sup_{t\in[0,{T\wedge\tau_J}]}\|u^M(t)\|^p_*\Big)\nonumber\\
&\leq&
\|u(0)\|^p_*+C_{p}E\int_0^{T\wedge\tau_J}\|u^M(s)\|^p_{*}ds+C_{\eta,N,p,\delta}E\Big(\sup_{t\in[0,T]}\|u^M(t)\|^p_\mathbb{V}\Big).\nonumber\\
\end{eqnarray*}
Let $\eta=\delta=\frac{1}{4p}$, for every $\e\in(0,1)$ to obtain
\begin{eqnarray}\label{eq 09}
E\Big(\sup_{s\in[0,{T\wedge\tau_J}]}\|u^M(s)\|^p_\mathbb{W}\Big)\leq C_{p,N}.
\end{eqnarray}
By Fatou's lemma, (\ref{eq 09}) implies (\ref{Eq G02 02}).
\end{proof}

Let $\mathbb{H}$ be a separable Hilbert space. Given $p>1$, $\beta\in(0,1)$, let $W^{\beta,p}([0,T];\mathbb{H})$ be the
 space of all $u\in L^p([0,T];\mathbb{H})$ such that
$$
\int_0^T\int_0^T\frac{\|u(t)-u(s)\|^p_\mathbb{H}}{|t-s|^{1+\beta p}}dtds<\infty
$$
endowed with the norm
$$
\|u\|^p_{W^{\beta,p}([0,T];\mathbb{H})}:=\int_0^T\|u(t)\|^p_{\mathbb{H}}dt+\int_0^T\int_0^T\frac{\|u(t)-u(s)\|^p_\mathbb{H}}{|t-s|^{1+\beta p}}dtds.
$$

The following result represents a variant of the criteria for compactness proved in \cite{Lions} (Sect. 5, Ch. I)
 and \cite{Temam 1983} (Sect. 13.3).
\begin{lem}\label{Compact}{\rm
Let $\mathbb{H}_0\subset \mathbb{H}\subset \mathbb{H}_1$ be Banach spaces, $\mathbb{H}_0$ and $\mathbb{H}_1$ reflexive, with compact embedding of $\mathbb{H}_0$ into $\mathbb{H}$.
For $p\in(1,\infty)$ and $\beta\in(0,1)$, let $\Lambda$ be the space
$$
\Lambda=L^p([0,T];\mathbb{H}_0)\cap W^{\beta,p}([0,T];\mathbb{H}_1)
$$
endowed with the natural norm. Then the embedding of $\Lambda$ into $L^p([0,T];\mathbb{H})$ is compact.
}\end{lem}

\begin{prp}\label{Prop Tight}
$\{u^{h^\e}\}$ is tight in $L^2([0,T];\mathbb{V})$.
\end{prp}

\begin{proof}

Note that
\begin{eqnarray}\label{Eq LDP}
u^{h^\e}(t)&=&u_0-\int_0^t\nu \widehat{A}u^{h^\e}(s)ds-\int_0^t\widehat{B}(u^{h^\e}(s),u^{h^\e}(s))ds\\
&+&
\int_0^t\widehat{F}(u^{h^\e}(s),s)ds+\sqrt{\epsilon}\int_0^t\widehat{G}(u^{h^\e}(s),s)dW(s)
+
\int_0^t\widehat{G}(u^{h^\e}(s),s)\dot{h^\e}(s) ds,\nonumber\\
&=&
J_1(t)+J_2(t)+J_3(t)+J_4(t)+J_5(t)+J_6(t).\nonumber
\end{eqnarray}


Using Lemma \ref{Lemma Esta1}, it is easy to show that
\begin{eqnarray}\label{Eq Tight 01}
E(\int_0^T\|u^{h^\e}(s)\|^2_\mathbb{W}ds)\leq C_{2,N},
\end{eqnarray}
where $C_{2,N}$ is a constant independent of $\e$.
We next prove
\begin{eqnarray}\label{4 36}
\sup_{\e\in(0,\e_0)}E(\|u^{h^\e}\|^2_{W^{\beta,2}([0,T],\mathbb{W}^*)})\leq C<\infty,
\end{eqnarray}
here $\e_0$ is the constant stated in Lemma \ref{Lemma Esta1}.
Noting that, for any $u\in\mathbb{W}$ and $v\in\mathbb{V}$,
\begin{eqnarray*}
(\widehat{A}u,v)_\mathbb{V}=((u,v)),
\end{eqnarray*}
we have
\begin{eqnarray}\label{Estation-A}
\|\widehat{A}u\|_\mathbb{V}
=
\sup_{\|v\|_\mathbb{V}\leq1}|(\widehat{A}u,v)_\mathbb{V}|
=
\sup_{\|v\|_\mathbb{V}\leq1}|((u,v))|
\leq
\|u\|\sup_{\|v\|_\mathbb{V}\leq1}\|v\|
\leq
\alpha^{-1/2}\|u\|.
\end{eqnarray}
Then
\begin{eqnarray}\label{Eq J2-01}
\|J_2(t)-J_2(s)\|^2_\mathbb{V}
&=&
\|\int_s^t\nu \widehat{A}u^{h^\e}(l)dl\|^2_\mathbb{V}
\leq
\int_s^t\|\nu \widehat{A}u^{h^\e}(l)\|^2_\mathbb{V}dl(t-s)\nonumber\\
&\leq&
\nu^2\alpha^{-1}\int_s^t\|u^{h^\e}(l)\|^2dl(t-s)
\leq
\nu^2\alpha^{-2}\sup_{l\in[0,T]}\|u^{h^\e}(l)\|^2_\mathbb{V}(t-s)^2.
\end{eqnarray}

For any $\beta\in(0,1/2)$, we have
\begin{eqnarray}\label{Eq J2}
E(\|J_2\|^2_{\mathbb{W}^{\beta,2}([0,T];\mathbb{V})})
&=&
E(\int_0^T\|J_2(s)\|^2_\mathbb{V}ds
+
\int_0^T\int_0^T\frac{\|J_2(t)-J_2(s)\|^2_\mathbb{V}}{|t-s|^{1+2\beta}}dsdt)\nonumber\\
&\leq&
TE(\sup_{s\in[0,T]}\|J_2(s)\|^2_\mathbb{V})
+
E\int_0^T\int_0^T\nu^2\alpha^{-2}\sup_{s\in[0,T]}\|J_2(s)\|^2_\mathbb{V}(t-s)^{1-2\beta}dsdt\nonumber\\
&\leq&
C_{\beta,T}E(\sup_{l\in[0,T]}\|u^{h^\e}(l)\|^2_\mathbb{V}).
\end{eqnarray}

By (\ref{Eq B-02}), we have
\begin{eqnarray*}
   \|J_3(t)-J_3(s)\|^2_{\mathbb{W}^*}
&=&
   \|\int_s^t\widehat{B}(u^{h^\e}(l),u^{h^\e}(l))dl\|^2_{\mathbb{W}^*}\\
&\leq&
   \int_s^t\|\widehat{B}(u^{h^\e}(l),u^{h^\e}(l))\|^2_{\mathbb{W}^*}dl(t-s)\\
&\leq&
 C\sup_{l\in[0,T]}\|u^{h^\e}(l)\|^4_{\mathbb{V}}(t-s)^2,
\end{eqnarray*}
which yields
\begin{eqnarray}\label{Eq J3}
E(\|J_3\|^2_{\mathbb{W}^{\beta,2}([0,T];\mathbb{W}^*)})
\leq
C_{\beta,T} E(\sup_{l\in[0,T]}\|u^{h^\e}(l)\|^4_\mathbb{W}).
\end{eqnarray}

For $J_6$, we have
\begin{eqnarray*}
   \|J_6(t)-J_6(s)\|^2_{\mathbb{V}}
&=&
   \|\int_s^t\widehat{G}(u^{h^\e}(l))\dot{h^\e}(l)dl\|^2_{\mathbb{V}}\\
&\leq&
    \int_s^t\|\widehat{G}(u^{h^\e}(l))\|^2_{\mathbb{V}^{\otimes m}}dl\int_s^t\|\dot{h^\e}(l)\|^2_{\mathbb{R}^m}dl\\
&\leq&
    CN (\sup_{l\in[0,T]}\|u^{h^\e}(l)\|^2_{\mathbb{V}})(t-s),
\end{eqnarray*}
which implies that
\begin{eqnarray}\label{Eq J6}
E(\|J_6\|^2_{\mathbb{W}^{\beta,2}([0,T];\mathbb{V})})
\leq
C_{\beta,N,T} E(\sup_{l\in[0,T]}\|u^{h^\e}(l)\|^2_\mathbb{V}).
\end{eqnarray}

By similar arguments, we also have
\begin{eqnarray}\label{Eq J4,5}
E(\|J_4\|^2_{\mathbb{W}^{\beta,2}([0,T];\mathbb{V})})+E(\|J_5\|^2_{\mathbb{W}^{\beta,2}([0,T];\mathbb{V})})
\leq
C_{\beta} E(\sup_{l\in[0,T]}\|u^{h^\e}(l)\|^2_\mathbb{V}).
\end{eqnarray}

Combining (\ref{Eq Tight 01}), (\ref{Eq J2}), (\ref{Eq J6}) and (\ref{Eq J4,5}) and (\ref{Estation-02}) in Lemma \ref{Lemma Esta1},
we obtain (\ref{4 36}).

Since the imbedding $\mathbb{W}\subset\mathbb{V}$ is compact, by Lemma \ref{Compact}, $$\Lambda=L^2([0,T],\mathbb{W})\cap\mathbb{W}^{\beta,2}([0,T],\mathbb{W}^*)$$
is compactly imbedded in $L^2([0,T],\mathbb{V})$. Denote $\|\cdot\|_\Lambda:=\|\cdot\|_{L^2([0,T],\mathbb{W})}+\|\cdot\|_{\mathbb{W}^{\beta,2}([0,T],\mathbb{W}^*)}$. Thus for any $L>0$,
$$K_L=\{u\in L^2([0,T],\mathbb{V}),\ \|u\|_\Lambda\leq L\}$$
is relatively compact in $L^2([0,T],\mathbb{V})$.

We have $$P(u^{h^\e}\not\in K_L)\leq P(\|u^{h^\e}\|_\Lambda\geq L)\leq \frac{1}{L}E(\|u^{h^\e}\|_\Lambda)\leq \frac{C}{L}.$$
Choosing sufficiently large constant $L$, we see that $\{u^{h^\e},\ \e>0\}$ is tight in $L^2([0,T],\mathbb{V})$.
\end{proof}

\vspace{4mm}
Since the imbedding $\mathbb{W}^{\beta,2}([0,T],\mathbb{W}^*)\subset C([0,T];\mathbb{W}^*)$ is compact, the following result is a
consequence of (\ref{4 36}).
\begin{prp}\label{Prop tight W}
$\{u^{h^\e}\}$ is tight in $C([0,T];\mathbb{W}^*)$.
\end{prp}

\begin{thm}\label{Thm condition 02}
For every fixed $N\in\mathbb{N}$, let $h^\e,\ h\in\mathcal{A}_N$ be such that $h^\e$ converges in distribution to $h$ as $\e\rightarrow0$.
Then
$$
\Gamma^\e\left(W(\cdot)+\frac{1}{\sqrt\e}\int_0^{\cdot}\dot h^\e(s)ds\right)\text{ converges in distribution to }\Gamma^0(\int_0^{\cdot}\dot h(s)ds)
$$
in $C([0,T];\mathbb{V})$ as $\e\rightarrow0$.
\end{thm}

\begin{proof}
Note that $u^{h^\e}=\Gamma^\e\left(W(\cdot)+\frac{1}{\sqrt\e}\int_0^{\cdot}\dot h^\e(s)ds\right)$. By Proposition \ref{Prop Tight} and Proposition \ref{Prop tight W}, we know that
$\{u^{h^\e}\}$ is tight in $L^2([0,T],\mathbb{V})\cap C([0,T],\mathbb{W}^*)$.

Let $(u,\ h,\ W)$ be any limit point of the tight family $\{(u^{h^\e},\ h^\e,\ W),\ \epsilon\in(0,\epsilon_0)\}$.
We must show that $u$ has the same law as $\Gamma^0(\int_0^{\cdot}\dot h(s)ds)$,
and actually $u^{h^\e}\Longrightarrow u$ in the smaller space $C([0,T];\mathbb{V})$.

Set
$$
\Pi=\Big(L^2([0,T],\ \mathbb{V})\cap C([0,T],\ \mathbb{W}^*),\ S_N,\ C([0,T],\ \mathbb{R}^m)\Big).
$$
By the Skorokhod representation theorem, there exit a stochastic basis
$(\Omega^1,\mathcal{F}^1,\{\mathcal{F}_t^1\}_{t\in[0,T]},\mathbb{P}^1)$ and, on this basis,
$\Pi$-valued random variables $(\widetilde{X}^\epsilon,\ \widetilde{h}^\e,\ \widetilde{W}^\epsilon)$ $(\widetilde{X},\ \widetilde{h},\ \widetilde{W})$
such that $(\widetilde{X}^\epsilon,\ \widetilde{h}^\e,\ \widetilde{W}^\epsilon)$ (respectively $(\widetilde{X},\ \widetilde{h},\ \widetilde{W})$) has the same law as $\{(u^{h^\e},\ h^\e,\ W),\ \epsilon\in(0,\epsilon_0)\}$ (respectively $(u,\ h,\ W)$), and
$(\widetilde{X}^\epsilon,\ \widetilde{h}^\e,\ \widetilde{W}^\epsilon)\rightarrow (\widetilde{X},\ \widetilde{h},\ \widetilde{W})$-$\mathbb{P}^1$ a.s. in $\Pi$.

From the equation satisfied by $(u^{h^\e},\ h^\e,\ W)$, we see that $(\widetilde{X}^\epsilon,\ \widetilde{h}^\e,\ \widetilde{W}^\epsilon)$ satisfies the following integral equation
\begin{eqnarray}\label{Eq X01}
\widetilde{X}^\epsilon(t)&=&u_0-\int_0^t\nu \widehat{A}\widetilde{X}^\epsilon(s)ds-\int_0^t\widehat{B}(\widetilde{X}^\epsilon(s),\widetilde{X}^\epsilon(s))ds\\
& &+
\int_0^t\widehat{F}(\widetilde{X}^\epsilon(s),s)ds+\sqrt{\epsilon}\int_0^t\widehat{G}(\widetilde{X}^\epsilon(s),s)d\widetilde{W}^\epsilon(s)
+
\int_0^t\widehat{G}(\widetilde{X}^\epsilon(s),s)\dot{\widetilde{h}^\e}(s) ds,\nonumber
\end{eqnarray}
and (see (\ref{Estation-02}))
\begin{eqnarray}\label{Estation-X01}
\sup_{\epsilon\in(0,\epsilon_0)}E(\sup_{s\in[0,T]}\|\widetilde{X}^\epsilon(s)\|^p_\mathbb{W})\leq C_p,\ \text{for any }2\leq p<\infty.
\end{eqnarray}

Using similar arguments as in the proof of Theorem 3.4 and Theorem 4.1 in \cite{RS-12}, we can show that $\widetilde{X}$ is the
unique solution of the following equation
\begin{eqnarray}\label{Eq X02}
\widetilde{X}(t)&=&u_0-\int_0^t\nu \widehat{A}\widetilde{X} (s)ds-\int_0^t\widehat{B}(\widetilde{X} (s),\widetilde{X} (s))ds\\
& &+
\int_0^t\widehat{F}(\widetilde{X} (s),s)ds
+
\int_0^t\widehat{G}(\widetilde{X} (s),s)\dot{\widetilde{h}} (s) ds.\nonumber
\end{eqnarray}

Finally, we will prove that
\begin{eqnarray}\label{Eq F01}
\lim_{\epsilon\rightarrow0}\sup_{t\in[0,T]}\|\widetilde{X}(t)-\widetilde{X}^\e(t)\|_\mathbb{V}=0,\ \ \ \mathbb{P}^1-a.s..
\end{eqnarray}

Let $v^\e(t)=\widetilde{X}^\e(t)-\widetilde{X}(t)$. Using $It\hat{o}$ formula, we have
\begin{eqnarray*}
&  &\|v^\e(t)\|^2_{\mathbb{V}}+2\int_0^t\Big(\nu \|v^\e(s)\|^2+\langle\widehat{B}(\widetilde{X}^\e(s),\widetilde{X}^\e(s))-\widehat{B}(\widetilde{X}(s),\widetilde{X}(s)),v^\e(s)\rangle\Big)ds\nonumber\\
&=&
   \int_0^t
       2\Big(
           (\widehat{F}(\widetilde{X}^\e(s),s)-\widehat{F}(\widetilde{X}(s),s),v^\e(s))_\mathbb{V}
            +
            \e\|\widehat{G}(\widetilde{X}^\epsilon(s),s)\|^2_{\mathbb{V}^{\otimes m}}
           \Big)ds\nonumber\\
  &  & +
        2\int_0^t\sqrt{\e}(\widehat{G}(\widetilde{X}^\e(s),s),v^\e(s))_\mathbb{V}d\widetilde{W}^\e(s)\nonumber\\
  &  &+
       2\int_0^t(\widehat{G}(\widetilde{X}^\e(s),s)\dot{\widetilde{h}^\e}(s)-\widehat{G}(\widetilde{X}(s))\dot{\widetilde{h}}(s),v^\e(s))_\mathbb{V}ds.
\end{eqnarray*}
Since
$$
\langle\widehat{B}(\widetilde{X}^\e(s),\widetilde{X}^\e(s))-\widehat{B}(\widetilde{X}(s),\widetilde{X}(s)),v^\e(s)\rangle
=
-\langle\widehat{B}(v^\e(s),v^\e(s)),\widetilde{X}(s)\rangle,
$$
by (\ref{Eq B-02}) and (\ref{F-02}), it follows that
\begin{eqnarray*}
&  &\|v^\e(t)\|^2_{\mathbb{V}}+2\nu\int_0^t\|v^\e(s)\|^2ds\nonumber\\
&\leq&
   C\int_0^t\|v^\e(s)\|^2_\mathbb{V}\|\widetilde{X}(s)\|_\mathbb{W}ds
   +
   C\int_0^t\|v^\e(s)\|^2_\mathbb{V}ds
   +
   C\e\int_0^t\|\widetilde{X}^\epsilon(s)\|^2_\mathbb{V}ds
           \nonumber\\
  &  & +
        2\int_0^t\sqrt{\e}(\widehat{G}(\widetilde{X}^\e(s),s),v^\e(s))_\mathbb{V}d\widetilde{W}^\e(s)\nonumber\\
  &  &+
       2\int_0^t(\widehat{G}(\widetilde{X}^\e(s),s)\dot{\widetilde{h}^\e}(s)-\widehat{G}(\widetilde{X}(s))\dot{\widetilde{h}}(s),v^\e(s))_\mathbb{V}ds.
\end{eqnarray*}
Since
\begin{eqnarray*}
&  &\int_0^t|(\widehat{G}(\widetilde{X}^\e(s),s)\dot{\widetilde{h}^\e}(s)-\widehat{G}(\widetilde{X}(s),s)\dot{\widetilde{h}}(s),v^\e(s))_\mathbb{V}|ds\\
&\leq&
   \int_0^t|(\widehat{G}(\widetilde{X}^\e(s),s)\dot{\widetilde{h}^\e}(s)-\widehat{G}(\widetilde{X}(s),s)\dot{\widetilde{h}^\e}(s),v^\e(s))_\mathbb{V}|ds\\
  & &+
   \int_0^t|(\widehat{G}(\widetilde{X}(s),s)\dot{\widetilde{h}^\e}(s)-\widehat{G}(\widetilde{X}(s),s)\dot{\widetilde{h}}(s),v^\e(s))_\mathbb{V}|ds\\
&\leq&
    C\int_0^t\|v^\e(s)\|^2_\mathbb{V}\|\dot{\widetilde{h}^\e}(s)\|_{\mathbb{R}^m}ds
  +
   C\int_0^t\|\widetilde{X}(s)\|_{\mathbb{V}}\|\dot{\widetilde{h}^\e}(s)-\dot{\widetilde{h}}(s)\|_{\mathbb{R}^m}\|v^\e(s)\|_\mathbb{V}ds\\
 &\leq&
  C\int_0^t\|v^\e(s)\|^2_\mathbb{V}\|\dot{\widetilde{h}^\e}(s)\|_{\mathbb{R}^m}ds
  +
   CN^{1/2}\sup_{s\in[0,T]}\|\widetilde{X}(s)\|_{\mathbb{V}}(\int_0^T\|v^\e(s)\|^2_\mathbb{V}ds)^{1/2},
\end{eqnarray*}
we have
\begin{eqnarray*}
&  &\|v^\e(t)\|^2_{\mathbb{V}}+2\nu\int_0^t\|v^\e(s)\|^2ds\nonumber\\
&\leq&
   C\int_0^t\|v^\e(s)\|^2_\mathbb{V}\|\widetilde{X}(s)\|_\mathbb{W}ds
   +
   C\int_0^t\|v^\e(s)\|^2_\mathbb{V}ds
   +
   C\e\int_0^t\|\widetilde{X}^\epsilon(s)\|^2_\mathbb{V}ds
           \nonumber\\
  &  & +
        2\int_0^t\sqrt{\e}(\widehat{G}(\widetilde{X}^\e(s),s),v^\e(s))_\mathbb{V}d\widetilde{W}^\e(s)\nonumber\\
  &  &+
       C\int_0^t\|v^\e(s)\|^2_\mathbb{V}\|\dot{\widetilde{h}^\e}(s)\|_{\mathbb{R}^m}ds
  +
   CN^{1/2}\sup_{s\in[0,T]}\|\widetilde{X}(s)\|_{\mathbb{V}}(\int_0^T\|v^\e(s)\|^2_\mathbb{V}ds)^{1/2}\\
&\leq&
   C_N\int_0^t\|v^\e(s)\|^2_\mathbb{V}\varphi^\e(s)ds+\Theta(\e,T),
\end{eqnarray*}
here
\begin{eqnarray}\label{Eq varphi}
\varphi^\e(s,\omega)=\|\widetilde{X}(s)\|_\mathbb{W}+1+\|\dot{\widetilde{h}^\e}(s)\|_{\mathbb{R}^m},
\end{eqnarray}
and
\begin{eqnarray}
\Theta(\e,T,\omega)
&=&
C\e\int_0^T\|\widetilde{X}^\epsilon(s)\|^2_\mathbb{V}ds
+
2\sup_{t\in[0,T]}|\int_0^t\sqrt{\e}(\widehat{G}(\widetilde{X}^\e(s),s),v^\e(s))_\mathbb{V}d\widetilde{W}^\e(s)|\nonumber\\
&&+
   CN^{1/2}\sup_{s\in[0,T]}\|\widetilde{X}(s)\|_{\mathbb{V}}(\int_0^T\|v^\e(s)\|^2_\mathbb{V}ds)^{1/2}.
\end{eqnarray}
By Gronwall' inequality,
\begin{eqnarray}\label{Eq gronwall}
    \sup_{t\in[0,T]}\|v^\e(t)\|^2_{\mathbb{V}}
&\leq&
    \Theta(\e,T)
    \exp\Big(\int_0^T\varphi^\e(s)ds\Big).
\end{eqnarray}

Noting that $\lim_{\e\rightarrow0}\widetilde{X}^\e=\widetilde{X}$ in $L^2([0,T],\mathbb{V})\ \mathbb{P}^1$-a.s.
and $\sup_{s\in[0,T]}\|\widetilde{X}(s,\omega)\|_\mathbb{W}\leq C(\omega)<\infty\ \mathbb{P}^1$-a.s.,
we have
\begin{eqnarray}\label{Eq es 01}
 \exp\Big(\int_0^T\varphi^\e(s,\omega)ds\Big)\leq C(\omega)<\infty\ \mathbb{P}^1-a.s.,
\end{eqnarray}
and
\begin{eqnarray}\label{Eq es 02}
 \lim_{\e\rightarrow0}\Theta(\e,T)
 =0,
      \ \ \ \mathbb{P}^1-a.s..
\end{eqnarray}
Hence
 \begin{eqnarray*}
  \sup_{t\in[0,T]}\|v^\e(t)\|^2_{\mathbb{V}}\rightarrow 0,\ \ \ \mathbb{P}^1\text{-}a.s..
 \end{eqnarray*}

\end{proof}
%

Replacing $\sqrt{\epsilon}\int_0^t\widehat{G}(u^{h^\e}(s),s)dW(s)$ with 0 in the proof of Proposition \ref{Prop Tight}, Proposition \ref{Prop tight W}
and Theorem \ref{Thm condition 02}, we have

\begin{thm}\label{Thm condition 01}
$\Gamma^0(\int_0^{\cdot}\dot g(s)ds)$ is a continuous mapping from $g\in S_N$ into $C([0,T];\mathbb{V})$, in particular,
$\{\Gamma^0(\int_0^{\cdot}\dot g(s)ds);\ g\in S_N\}$ is a compact subset of $C([0,T],\mathbb{V})$.
\end{thm}

\def\refname{ References}

\end{document}